\documentclass[11pt]{amsart}

\usepackage{latexsym,rawfonts}
\usepackage{amsfonts,amssymb}
\usepackage{amsmath,amsthm}

\usepackage[plainpages=false]{hyperref}
\usepackage{graphicx}
\usepackage{color}

\newtheorem{theorem}{Theorem}

\newtheorem{lemma}{Lemma}

\theoremstyle{definition}

\newtheorem{remark}{Remark}

\newtheorem{condition}{Condition}

\newcommand{\beq}{\begin{equation}}
\newcommand{\eeq}{\end{equation}}
\newcommand{\beqs}{\begin{eqnarray*}}
\newcommand{\eeqs}{\end{eqnarray*}}
\newcommand{\beqn}{\begin{eqnarray}}
\newcommand{\eeqn}{\end{eqnarray}}
\newcommand{\beqa}{\begin{array}}
\newcommand{\eeqa}{\end{array}}

\newcommand{\R}{\mathbb{R}}

\newcommand{\set}[1]{\left\{#1\right\}}

\newcommand{\uS}{\mathbb{S}^{n-1}}

\newcommand{\MA}{Monge-Amp\`ere }

\begin{document}

\title{$C^{1, 1}$ regularity for solutions to the degenerate $L_p$ Dual Minkowski problem}

\author{Li Chen}
\address{Faculty of Mathematics and Statistics, Hubei Key Laboratory of Applied Mathematics, Hubei University,  Wuhan 430062, P.R. China}
\email{chenli@hubu.edu.cn}

\author{Qiang Tu}
\address{Faculty of Mathematics and Statistics, Hubei Key Laboratory of Applied Mathematics, Hubei University,  Wuhan 430062, P.R. China}
\email{qiangtu@hubu.edu.cn}

\author{Di Wu}
\address{Faculty of Mathematics and Statistics, Hubei Key Laboratory of Applied Mathematics, Hubei University,  Wuhan 430062, P.R. China}
\email{wudi19950106@126.com}

\author{Ni Xiang}
\address{Faculty of Mathematics and Statistics, Hubei Key Laboratory of Applied Mathematics, Hubei University,  Wuhan 430062, P.R. China}
\email{nixiang@hubu.edu.cn}
\thanks{This research was supported by the National Natural Science Foundation of China No.11971157.}


\date{}

\begin{abstract}
In this paper, we study $C^{1, 1}$ regularity for solutions to the
degenerate $L_p$ Dual Minkowski problem. Our proof is motivated by
the idea of Guan and Li's work on $C^{1,1}$ estimates for solutions
to the Aleksandrov problem.
\end{abstract}



\maketitle
\vskip4ex

\section{Introduction}

The classic Minkowski problem is one of the cornerstones in the
Brunn-Minkowski theory of convex bodies. Its solution has many
applications in various fields of geometry and analysis, see
\cite{Sch13} for an overview. An important counterpart of the
classic Minkowski problem is the famous Aleksandrov problem
characterzing the integral Gauss curvature, which is introduced and
completely solved by Aleksandrov \cite{Al}.

Oliker \cite{Ol83} has shown that there is a PDE associated with the
Aleksandrov problem:
\begin{equation} \label{Al}
\frac{h}{(|\nabla h|^2+h^2)^{\frac{n}{2}}}\det(\nabla^2h +hI) =f
\text{ on } \uS,
\end{equation}
where $f$ is a given positive function  defined on the unit sphere
$\uS$, $h$ is an unknown function on $\uS$. Here $\nabla$ is the
covariant derivative with respect to an orthonormal frame on $\uS$,
$I$ is the unit matrix of order $n-1$. When $f$ is a smooth positive
function, the solution to the Aleksandrov problem \eqref{Al} is
smooth, see \cite{Po73} and \cite{Ol83}. When $f$ is a smooth, but
only nonnegative, Guan and Li \cite{Guan-Li} have shown solutions to the
Aleksandrov problem \eqref{Al} are at least $C^{1, 1}$ in dimension
$n=3, 4$. For higher dimensions, they have obtain the same conclusion under
some further hypothesis on $f$.

Recently, the $L_p$ dual Minkowski problem was  introduced
in \cite{LYZ.Adv.329-2018.85}. For a given positive function $f$
defined on the unit sphere $\uS$, the \emph{$L_p$ dual Minkowski
problem} is concerned with the solvability of the \MA type equation
\begin{equation} \label{Lp-A}
\frac{h^{1-p}}{(|\nabla h|^2+h^2)^{\frac{n-q}{2}}}\det(\nabla^2h
+hI) =f \text{ on } \uS,
\end{equation}
for some support function $h$ of a hypersurface $M$ in the Euclidean
space $\R^n$ enclosing the origin. The $L_p$ dual Minkowski problem
unifies the Aleksandrov problem ($p=0, q=0$), the dual Minkowski
problem ($p=0$) and the $L_p$-Minkowski problem ($q=n$).  The dual
Minkowski problem was first proposed by Huang, Lutwak, Yang and
Zhang in their recent groundbreaking work
\cite{HLYZ.Acta.216-2016.325} and then followed by
\cite{BHP.JDG.109-2018.411,
  HP.Adv.323-2018.114,
  HJ.JFA.277-2019.2209,
  LSW.JEMSJ.22-2020.893,
  Zha.CVPDE.56-2017.18,
  Zha.JDG.110-2018.543}.
The $L_p$-Minkowski problem was introduced by Lutwak
\cite{Lut.JDG.38-1993.131} in 1993 and has been extensively studied
since then; see e.g.
\cite{
  BLYZ.JAMS.26-2013.831,
  CLZ.TAMS.371-2019.2623,
  Zhu.Adv.262-2014.909}
for the logarithmic Minkowski problem ($p=0$),
\cite{JLW.JFA.274-2018.826,
  JLZ.CVPDE.55-2016.41,
  Lu.SCM.61-2018.511,
  Lu.JDE.266-2019.4394,
  LW.JDE.254-2013.983,
  Zhu.JDG.101-2015.159}
for the centroaffine Minkowski problem ($p=-n$), and
\cite{CW.Adv.205-2006.33,
  HLYZ.DCG.33-2005.699,
  LYZ.TAMS.356-2004.4359,
  Sta.Adv.167-2002.160}
for other cases of the $L_p$-Minkowski problem. For the general
$L_p$ dual Minkowski problem, much progress has already been made in
\cite{BF.JDE.266-2019.7980,
  CHZ.MA.373-2019.953,
  CCL,
  HLYZ.JDG.110-2018.1,
  HZ.Adv.332-2018.57,
  LLL}.

When $f$ is a smooth positive function, the solution to the $L_p$
dual Minkowski problem \eqref{Lp-A} is smooth provided either $p>q$
or $pq\geq 0$ and $f$ is even (see \cite{HZ.Adv.332-2018.57} and
\cite{CHZ.MA.373-2019.953}). It is natural to ask that when $f$ is
smooth, but only nonnegative, are the solutions to the $L_p$ dual
Minkowski problem \eqref{Lp-A} smooth? In this case, we encounter
with certain degenerate \MA type equation. Regularity of solutions
to degenerate \MA type equations has been investigated in \cite{An,
Kr89, Kr90, Kr95, Ca86, Guan97, Guan99, Hong94, Tr91, Tr83, Tr84,
Le15, Li15} and the references therein. The global $C^{1,1}$
regularity of degenerate \MA type equations has been obtained in
\cite{Guan99} and one cannot expect regularity higher than $C^{1,1}$
in general \cite{Wang95}.

In this paper, we study $C^{1, 1}$ regularity for solutions to the
$L_p$ Dual Minkowski problem \eqref{Lp-A}  when $f$ is smooth
enough, but only nonnegative. For low dimensions $n=3, 4$, we can obtain the following result.
\begin{theorem}\label{thm}
Suppose $p>q>0$ and $n=3$ or $n=4$. Let $f$ be a smooth, nonnegative,
nonzero and even function on
$\uS$. Then, there exists a generalized solution $h \in C^{1, 1}(\uS)$ satisfying the
equation \eqref{Lp-A}.
\end{theorem}

For higher dimensions, we need some additional hypothesis.
To statement our result we recall the
following Condition \ref{CI} which was introduced by Guan-Li in
\cite{Guan-Li}.

\begin{condition}\label{CI}
$f\in C^{2}(\uS)$ is nonnegative and there exist a constant $A$
such that
\begin{eqnarray}\label{I-1}
|\nabla (f^{\frac{1}{n-2}})|\leq A \quad \text{on} \ \uS,
\end{eqnarray}
and
\begin{eqnarray}\label{I-2}
\Delta (f^{\frac{1}{n-2}})\geq -A  \quad \text{on} \ \uS.
\end{eqnarray}
\end{condition}

It is clearly for nonnegative function $f\in C^{2}(\uS)$, the
condition \eqref{I-1} is equivalent to
\begin{eqnarray*}
|\nabla f(x)|\leq(n-2)A f^{1-\frac{1}{n-2}}(x),
\end{eqnarray*}
and  the condition
\eqref{I-2} is equivalent to
\begin{eqnarray*}
f(x)\Delta f(x)-\frac{n-2}{n-3}|\nabla f(x)|^2\geq-(n-2)A
f^{2-\frac{1}{n-2}}(x).
\end{eqnarray*}

We also introduce the following Condition \ref{CII}.

\begin{condition}\label{CII}
$f\in C^{2}(\uS)$ is nonnegative, $q<2$ and there exist some
constants $A$ such that
\begin{eqnarray}\label{II}
f \Delta f-\frac{3-q}{2-q}|\nabla f|^2
\geq-Af^{2-\frac{1}{n-2}} \quad \text{on} \ \uS.
\end{eqnarray}
\end{condition}

When $q=0$, Condition \ref{CII} was introduced  by Guan-Li in
\cite{Guan-Li}. Clearly, if $f$ satisfies Condition \ref{CII}, then
$f$ satisfies the condition \eqref{I-2}.

Using the idea in \cite{Guan-Li}, we can prove the following
theorem.

\begin{theorem}\label{thm-1}
Suppose $p>q>0$, $f$ is a nonnegative, nonzero and even function on
$\uS$. Then, there exists a generalized solution $h \in C^{1, 1}(\uS)$ satisfying the
equation \eqref{Lp-A}, provided either $f \in C^{2, \alpha}(\uS)\
(0<\alpha<1)$ satisfies Condition \ref{CI} or Condition \ref{CII}.
\end{theorem}

\begin{remark}
In \cite{Guan-Li}, Guan-Li showed that for $n=3$, all nonnegative
function $f \in C^{1, 1}(\mathbb{S}^2)$ satisfied Condition
\ref{CI} and all nonnegative function $f \in C^{3, 1}(\mathbb{S}^3)$
satisfied Condition \ref{CI} for $n=4$. Thus, Theorem \ref{thm} is just
a direct corollary of Theorem \ref{thm-1}.
\end{remark}

\begin{remark}
The condition that $f$ is nonzero in Theorem \ref{thm-1} is
necessary, otherwise we can see that $h\equiv +\infty$ by Lemma
\ref{C0-1}. When $f$ is a smooth positive function, the existence of
smooth solutions to the $L_p$ dual Minkowski problem \eqref{Lp-A} is
attained in \cite{HZ.Adv.332-2018.57} provided $p>q$, thus it is a
natural question to ask whether we can drop the assumptions that
$q>0$ and $f$ is an even function in Theorem \ref{thm-1}.
\end{remark}

The organization of the paper is as follows. $C^0$, $C^1$ and $C^2$
estimates are given in Sect. 2. In Sect. 3 we prove Theorem
\ref{thm-1}.

\section{A priori estimates}

In this section, we will establish $C^0$, $C^1$ and $C^2$ estimates
for solutions to \eqref{Lp-A}. The key is that those estimates must
be independent of $\min_{\uS} f$.

\subsection{Basic properties of convex hypersurfaces}

We first recall some basic properties of convex hypersurfaces in
$\R^n$; see \cite{Urb.JDG.33-1991.91} for details. Let $M$ be a
smooth, closed, uniformly convex hypersurface in $\R^n$ enclosing
the origin. The support function $h$ of $M$ is defined as
\begin{equation} \label{h0}
h(x) := \max_{y\in M} \langle y,x \rangle, \quad \forall\, x\in\uS,
\end{equation}
where $\langle \cdot,\cdot \rangle$ is the standard inner product in
$\R^n$.

The convex hypersurface $M$ can be recovered by its support function
$h$. In fact, writing the Gauss map of $M$ as $\nu_M$, we
parametrize $M$ by $X : \uS\to M$ which is given as
\begin{equation*}
X(x) =\nu_M^{-1}(x), \quad \forall\,x\in\uS.
\end{equation*}
Note that $x$ is the unit outer normal vector of $M$ at $X(x)$. On
the other hand, one can easily check that the maximum in the
definition \eqref{h0} is attained at $y=\nu_M^{-1}(x)$, namely
\begin{equation} \label{h}
  h(x) = \langle x, X(x)\rangle, \quad\forall\, x \in \mathbb{S}^{n-1}.
\end{equation}
Let $\nabla$ be the corresponding connection on $\mathbb{S}^{n-1}$.
Differentiating the both sides of \eqref{h}, we have
\begin{equation*}
  \nabla_{i} h = \langle \nabla_{i}x, X(x)\rangle + \langle x, \nabla_{i}X(x)\rangle.
\end{equation*}
Since $\nabla_{i}X(x)$ is tangent to $M$ at $X(x)$, there is
\begin{equation*}
  \nabla_{i} h = \langle \nabla_{i}x, X(x)\rangle,
\end{equation*}
which together with \eqref{h} implies that
\begin{equation*}\label{Xh}
X(x) = \nabla h(x) + h(x)x, \quad \forall\,x\in\uS.
\end{equation*}
The radial function $\rho$ of the convex hypersurface $M$ is defined
as
\begin{equation*}
\rho(u) :=\max\set{\lambda>0 : \lambda u\in M}, \quad\forall\,
u\in\uS.
\end{equation*}
Note that $\rho(u)u\in M$. If we connect $u$ and $x$ through the
following equality:
\begin{equation}\label{eq:8}
\rho(u)u =X(x) =\nabla h(x) +h(x)x.
\end{equation}
By virtue of \eqref{eq:8}, there is
\begin{equation*}
\rho^2=|\nabla h|^2+h^2,
\end{equation*}
which implies that
\begin{equation}\label{G-1}
|\nabla h|\leq \rho.
\end{equation}
By \eqref{h} and \eqref{eq:8} , we have
\begin{equation}\label{G-2}
\max_{\uS}h=\max_{\uS}\rho.
\end{equation}

\subsection{$C^0$ estimate and the gradient estimate}

Now, we begin to prove $C^0$ estimate.

\begin{lemma}\label{C0-1}
Assume $p>q$ and $f$ is a nonnegative and continuous function. Let $h
\in C^{2}(\uS)$ be a solution to \eqref{Lp-A}, then there exists a
positive constant $C$ depending on $p$, $q$ and $\max_{\uS} f$ such
that
\begin{eqnarray*}
\min_{\uS}h\geq C.
\end{eqnarray*}
\end{lemma}

\begin{proof}
Assume $h$ attains its minimum at $x_0$, using the equation
\eqref{Lp-A}, it is straightforward to see that
\begin{eqnarray*}
h^{q-p}(x_0)\leq \max_{  \uS} f.
\end{eqnarray*}
 Note that $p>q$, thus
\begin{eqnarray*}
h(x_0)\geq \frac{1}{[\max_{\uS} f]^{p-q}}.
\end{eqnarray*}
So, we complete the proof.
\end{proof}

\begin{lemma}\label{C0-2}
Assume $p>q>0$ and $f$ is a continuous, even and nonzero function.
Let $h \in C^{2}(\uS)$ be a solution to \eqref{Lp-A}, then there
exists a positive constant $C$ depending on $p$, $q$, $n$ and
$\max_{\uS} f$ such that
\begin{eqnarray*}
\max_{\uS}h\leq C.
\end{eqnarray*}
Thus,
\begin{eqnarray*}\label{Ge}
\max_{\uS}|\nabla h|\leq C,
\end{eqnarray*}
where $C$ is a positive constant depending on $p$, $q$, $n$ and
$\max_{\uS} f$.
\end{lemma}

\begin{proof}
Write
\begin{equation*}
h_{max}=\max_{x\in\uS} h(x)=h(x_0)
\end{equation*}
for some $x_0\in\uS$. Note that $f$ is even, so is $h$. Thus, we have
by the definition of support function that
\begin{equation*}
h(x)\geq h_{max}|\langle x,x_0 \rangle|, \quad \forall x\in\uS.
\end{equation*}
Thus,
\begin{eqnarray}\label{0-1}
\int_{\uS}h^p(x) f(x)dx\geq h_{max}^{p}\int_{\uS}|\langle x,x_0
\rangle|^pf(x)dx.
\end{eqnarray}
Since $f$ is nonzero, $\max_{\uS}f>0$. Assume $f$ attains its
maximum at $y$, there exist a ball $B(y, \delta)\subset \uS$ such
that
\begin{eqnarray}\label{0-2}
f(x)\geq \frac{1}{2}\max_{\uS}f, \quad \forall x \in B(y, \delta).
\end{eqnarray}
Substituting \eqref{0-2} into \eqref{0-1}, we have
\begin{eqnarray}\label{0-3}
\int_{\uS}h^p(x) f(x)dx&\geq& \frac{1}{2}h_{max}^{p}\max_{\uS}f
\int_{B(y, \delta)}|\langle x,x_0 \rangle|^pdx\nonumber\\&\geq&
\frac{1}{2}h_{max}^{p}\max_{\uS}f \min_{z\in \uS}\int_{B(y,
\delta)}|\langle x,z \rangle|^pdx\nonumber\\&\geq& C h_{max}^{p}.
\end{eqnarray}
Using the equation \eqref{Lp-A}, it is straightforward to see that
\begin{eqnarray}\label{0-5}
\int_{\uS}h^p(x) f(x)dx&=&\int_{\uS}\rho^q(u)du\nonumber\\& \leq&
|\uS| (\max_{\uS}\rho)^q\nonumber\\&=&|\uS| h^{q}_{max},
\end{eqnarray}
where we use \eqref{G-2} to get the last equality and $|\uS|$ is the
volume of $\uS$. Thus, combining \eqref{0-3}, \eqref{0-5} and the fact that $p>q$, we
obtain
\begin{eqnarray*}
h_{max}\leq C.
\end{eqnarray*}
Thus, we get the upper bound of $h$. The gradient estimate follows
from the upper bound of $h$, \eqref{G-1} and \eqref{G-2}
consequently.
\end{proof}

\subsection{$C^2$ estimate}

First, we give some notations. For a $(0, 2)$ tensor field
$b=\{b_{ij}\}$ on $\uS$, the coordinate expression of its covariant
derivative $\nabla b$ and the second covariant derivative
$\nabla^2b$ are denoted by
\begin{eqnarray*}
\nabla b=(b_{ij; k}), \quad \nabla^2 b=(b_{ij; kl}).
\end{eqnarray*}
However, the coordinate expression of the covariant differentiation
will be denoted by indices without semicolons, e.g. $$h_{i}, \quad
h_{ij} \quad \mbox{or} \quad h_{ijk}$$ for a function $h: \uS
\rightarrow \mathbb{R}$. To prove the $C^2$ estimate, we first
recall a simple algebraic inequality.

\begin{lemma}
For any $n-1$ real number $a_1, a_2, ..., a_{n-1}$ satisfying
\begin{eqnarray*}
\min\{a_1, a_2, ..., a_{n-1}\}\leq 0 \quad \mbox{and} \quad \max\{a_1, a_2, ..., a_{n-1}\}\geq 0,
\end{eqnarray*}
then we have the following inequality
\begin{eqnarray}\label{Alg}
\sum_{i=1}^{n-1}a^{2}_{i}\geq \frac{1}{n-2}(\sum_{i=1}^{n-1}a_{i})^2.
\end{eqnarray}
\end{lemma}

\begin{proof}
Without loss of generality, we assume
\begin{eqnarray*}
a_1\leq a_2\leq\cdot\cdot\cdot \leq a_s\leq 0\leq a_{s+1}\leq
\cdot\cdot\cdot\leq a_{n-1}, \quad 2\leq s\leq n-2,
\end{eqnarray*}
and
\begin{eqnarray*}
|a_1+a_2+\cdot\cdot\cdot+a_s|\geq a_{s+1}+\cdot\cdot\cdot+ a_{n-1}.
\end{eqnarray*}
Thus, using Cauchy-Schwartz inequality, we have
\begin{eqnarray*}
\sum_{i=1}^{n-1}a^{2}_{i}\geq\sum_{i=1}^{s}a^{2}_{i}\geq
\frac{1}{s}(\sum_{i=1}^{s}a_{i})^2\geq
\frac{1}{n-2}(\sum_{i=1}^{n-1}a_{i})^2.
\end{eqnarray*}
So, we complete the proof.
\end{proof}

\begin{lemma}\label{C2}
Let $h \in C^{4}(\uS)$ be a solution to \eqref{Lp-A}. Then there
exists a positive constant $C$ depending on $n$, $\max_{\uS} h$,
$\min_{\uS} h$, $\max_{\uS} |\nabla h|$, $A$ and $\max_{\uS} f$ such
that
\begin{eqnarray*}
|h|_{C^2(\uS)}\leq C
\end{eqnarray*}
provided $f \in C^2(\uS)$ satisfies either \textbf{Condition \ref{CI}} or
\textbf{Condition \ref{CII}}.
\end{lemma}

\begin{proof}
Set $b_{ij}=h_{ij}+h\delta_{ij}$ and $b=\{b_{ij}\}$, let
\begin{equation}\label{ht}
H(x)=\mathrm{tr} b=(n-1)h(x)+\Delta h,
\end{equation}
which is clearly nonnegative since the matrix $b=\nabla^2h +hI$ is
nonnegative definite. Then, there exists a point $x_0 \in \uS$
such that
\begin{equation*}
H(x_0)=\max_{x \in \uS}H(x).
\end{equation*}
If $H(x_0)\leq 1$, then our result holds. So, we assume $H(x_0)\geq 1$.
By choosing a suitable orthonormal frame, we may assume
$\{h_{ij}(x_0)\}$ is diagonal. Then,
\begin{equation}\label{D-H-1}
0=\nabla_iH(x_0)=\sum_{k}b_{kk;i}
\end{equation}
and
\begin{equation}\label{D-H-2}
\nabla_i\nabla_iH(x_0)=\sum_{k}b_{kk; ii}\leq 0.
\end{equation}
Since $\{b_{ij}\}$ is nonnegative definite, its inverse matrix
$\{b^{ij}\}$ is also nonnegative definite. Thus, we have at $x_0$
\begin{eqnarray}\label{C2-1}
0&\geq& b^{ii}H_{ii}\nonumber\\&=&b^{ii}\sum_{k}b_{kk;
ii}\nonumber\\&=&b^{ii}\Big(\sum_{k}b_{ii; k
k}-(n-1)b_{ii}+\sum_{k}b_{kk}\Big)\nonumber\\&=&\sum_{k}b^{ii}b_{ii;
k k}-(n-1)^2+\sum_{i}b^{ii}H,
\end{eqnarray}
where we use the Ricci identity
\begin{eqnarray*}
b_{ii; kk}=b_{kk; ii}-b_{kk}+b_{ii}.
\end{eqnarray*}
We rewrite the equation \eqref{Lp-A} as
\begin{equation} \label{Lp-A-r}
\mathrm{det}^{\frac{1}{n-2}}(\nabla^2h +hI)=\varphi^{\frac{1}{n-2}},
\end{equation}
where
\begin{equation*}
\varphi=\frac{(|\nabla h|^2+h^2)^{\frac{n-q}{2}}}{h^{1-p}}f.
\end{equation*}
Differentiating \eqref{Lp-A-r} twice, we can obtain at $x_0$
\begin{eqnarray} \label{C2-2}
\sum_{k}b^{ij}b_{ij; kk}&=&b^{ik}b^{lj}\sum_{p}b_{ij; p}b_{kl; p}
-\frac{1}{n-2}b^{ij}b^{kl}\sum_{p}b_{ij; p}b_{kl; p}\nonumber\\&&
+\varphi^{-2}\Big[\varphi \Delta \varphi-\frac{n-3}{n-2}|\nabla
\varphi|^2\Big] \nonumber\\&=&b^{ii}b^{jj}\sum_{p}(b_{ij;
p})^2-\frac{1}{n-2}b^{ij}b^{kl}\sum_{p}b_{ij; p}b_{kl;
p}\nonumber\\&& +\varphi^{-2}\Big[\varphi \Delta
\varphi-\frac{n-3}{n-2}|\nabla \varphi|^2\Big]
\nonumber\\&\geq&\sum_{i}\sum_{p}(b^{ii}b_{ii; p})^2
-\frac{1}{n-2}\sum_{p}(\sum_{i}b^{ii}b_{ii; p})^2\nonumber\\&&
+\varphi^{-2}\Big[\varphi \Delta \varphi-\frac{n-3}{n-2}|\nabla
\varphi|^2\Big] \nonumber\\&\geq&\varphi^{-2}\Big[\varphi \Delta
\varphi-\frac{n-3}{n-2}|\nabla \varphi|^2\Big],
\end{eqnarray}
where we use the inequality \eqref{Alg} to get the last inequality by noticing that $\sum_{i}b_{ii; k}(x_0)=0$.
Substituting \eqref{C2-2} into \eqref{C2-1}, we arrive at $x_0$
\begin{eqnarray}\label{C2-1-1}
0&\geq&\varphi^{-2}\Big[\varphi \Delta \varphi-\frac{n-3}{n-2}|\nabla \varphi|^2\Big]-(n-1)^2+H\sum_{i}b^{ii}.
\end{eqnarray}
Set $\psi=\frac{(|\nabla h|^2+h^2)^{\frac{n-q}{2}}}{h^{1-p}}$, thus
$\varphi=\psi f$. It follows consequently
\begin{eqnarray}\label{C2-3}
&&\varphi^{-2}\Big[\varphi \Delta \varphi-\frac{n-3}{n-2}|\nabla \varphi|^2\Big]
\nonumber\\&=&f^{-2}\Big[f \Delta f-\frac{n-3}{n-2}|\nabla f|^2\Big]+\frac{2}{n-2}f^{-1}
\psi^{-1}\nabla f\nabla \psi\nonumber\\&&+\psi^{-2}\Big[\psi \Delta \psi-\frac{n-3}{n-2}|\nabla \psi|^2\Big].
\end{eqnarray}
Differentiating $\psi$ twice, we have
\begin{eqnarray}\label{C2-4}
\nabla_i\psi&=&(n-q)\frac{(|\nabla
h|^2+h^2)^{\frac{n-q-2}{2}}}{h^{1-p}}\sum_{k}h_{ki}h_k\nonumber\\&&+(n-q)\frac{(|\nabla
h|^2+h^2)^{\frac{n-q-2}{2}}}{h^{1-p}}hh_i+(p-1)\frac{(|\nabla
h|^2+h^2)^{\frac{n-q}{2}}}{h^{2-p}}h_i\nonumber\\&=&(n-q)\frac{(|\nabla
h|^2+h^2)^{\frac{n-q-2}{2}}}{h^{1-p}}h_{ii}h_i+C(h, h^{-1}, \nabla
h),
\end{eqnarray}
\begin{eqnarray}\label{C2-5}
|\nabla\psi|^2&=&(n-q)^2\frac{(|\nabla
h|^2+h^2)^{n-q-2}}{h^{2(1-p)}}\sum_{i}h_{ii}^{2}h^{2}_{i}\nonumber\\&&+C(h,
h^{-1}, \nabla h)*b,
\end{eqnarray}
and
\begin{eqnarray}\label{C2-6}
\psi\Delta\psi&=&(n-q)\frac{(|\nabla
h|^2+h^2)^{n-q-2}}{h^{2(1-p)}}\Big[(|\nabla
h|^2+h^2)h_{ii}^{2}+(n-q-2)\sum_{i}h_{ii}^{2}h^{2}_{i}\Big]\nonumber\\&&+C(h,
h^{-1}, \nabla h)*b\nonumber\\&\geq&(n-q)(n-q-1)\frac{(|\nabla
h|^2+h^2)^{n-q-2}}{h^{2(1-p)}}\sum_{i}h_{ii}^{2}h^{2}_{i}-C(h,h^{-1}, \nabla h)H,
\end{eqnarray}
where $C(h, h^{-1}, \nabla h)$ denotes some quantity depending on
$h, h^{-1}, \nabla h$ and may change from line to line.
$C(h, h^{-1}, \nabla h)*b$ denotes some quantity linear in $b$ with
coefficients depending on $h, h^{-1}, \nabla h$.

Combining \eqref{C2-5} with \eqref{C2-6}, we can arrive
\begin{eqnarray} \label{C2-7}
\psi^{-2}\Big[\psi \Delta \psi-\frac{n-3}{n-2}|\nabla
\psi|^2\Big]&\geq& \psi^{-2}\frac{2-q}{(n-q)(n-2)}|\nabla
\psi|^2\nonumber\\&&-C(h, h^{-1}, \nabla h) H.
\end{eqnarray}
Using Cauchy-Schwartz inequality, we obtain
\begin{eqnarray}\label{C2-8}
\frac{2}{n-2}f^{-1}
\psi^{-1}\nabla f\nabla \psi&\leq&\frac{n-q}{(n-2)(2-q)}f^{-2}
|\nabla f|^2\nonumber\\&&+\frac{2-q}{(n-q)(n-2)}\psi^{-2}|\nabla \psi|^2.
\end{eqnarray}
Substituting \eqref{C2-7} and \eqref{C2-8} into \eqref{C2-3}, it yields
\begin{eqnarray}\label{C2-9}
&&\varphi^{-2}\Big[\varphi \Delta \varphi-\frac{n-3}{n-2}|\nabla \varphi|^2\Big]
\nonumber\\&\geq&f^{-2}\Big[f \Delta f-\frac{3-q}{2-q}|\nabla f|^2\Big]-C(h, h^{-1}, \nabla h) H.
\end{eqnarray}
Putting \eqref{C2-9} into \eqref{C2-1-1}, we arrive at $x_0$ due to
\textbf{Condition \ref{CII}}
\begin{eqnarray}\label{C2-10}
0&\geq&-Af^{-\frac{1}{n-2}}-(n-1)^2+H\sum_{i}b^{ii}-C_(h,
h^{-1}, \nabla h)H.
\end{eqnarray}
Now we need to estimate $\sum_{i}b^{ii}$. Without loss of generality, we assume
\begin{eqnarray*}
b_{11}\leq b_{22}\leq \cdot\cdot\cdot \leq b_{n-1\ n-1}.
\end{eqnarray*}
It follows that $b_{n-1 \ n-1}\geq \frac{H}{n-1}$. Thus,
\begin{eqnarray}\label{b}
\sum_{i}b^{ii}&\geq& \sum_{i=1}^{n-2}b^{ii}\geq(n-2)\Big(\prod_{i=1}^{n-2}
b^{ii}\Big)^{\frac{1}{n-2}}\nonumber\\&=&
(n-2)\Big(\frac{b_{n-1 \ n-1}}{\mathrm{det}\ b}\Big)^{\frac{1}{n-2}}\nonumber\\&\geq&C(n)H^{\frac{1}{n-2}}
(\mathrm{det} \ b)^{-\frac{1}{n-2}}.
\end{eqnarray}
Plugging the above inequality into \eqref{C2-10}, we have
\begin{eqnarray*}
(n-1)^2f^{\frac{1}{n-2}}+A&\geq&C(n)H^{1+\frac{1}{n-2}}\Big[\frac{f}{\mathrm{det} \ b}\Big]^{\frac{1}{n-2}}-f^{\frac{1}{n-2}}C(h,
h^{-1}, \nabla h)H\nonumber\\&\geq&C(n, h,
h^{-1}, \nabla h)H^{1+\frac{1}{n-2}}-f^{\frac{1}{n-2}}C(h,
h^{-1}, \nabla h)H.
\end{eqnarray*}
Thus, we conclude from above
\begin{eqnarray*}
H\leq C(n, h,
h^{-1}, \nabla h).
\end{eqnarray*}
This gives an upper bound of $\max_{\mathbb{S}^n}H$.

Next, we will prove an upper bound of $\max_{\mathbb{S}^n}H$ when
$f$ satisfies\textbf{ Condition \ref{CI}}, it follows by
\eqref{C2-4} and \eqref{I-1}
\begin{eqnarray}\label{C2-12}
\frac{2}{n-2}f^{-1}
\psi^{-1}|\nabla f\nabla \psi|&\leq&2 (n-q)Af^{-\frac{1}{n-2}}H\\
\nonumber&& +2Af^{-\frac{1}{n-2}}C(h,
h^{-1}, \nabla h).
\end{eqnarray}
Substituting \eqref{C2-12} and \eqref{I-2} into \eqref{C2-3}, we
have
\begin{eqnarray*}\label{C2-13}
&&\varphi^{-2}\Big[\varphi \Delta \varphi-\frac{n-3}{n-2}|\nabla \varphi|^2\Big]
\nonumber\\&\geq&-\frac{A}{n-2}f^{-\frac{1}{n-2}}-2(n-q)Af^{-\frac{1}{n-2}}C(h,
h^{-1}, \nabla h)H\nonumber\\&&-2Af^{-\frac{1}{n-2}}C(h,
h^{-1}, \nabla h)-C(h,
h^{-1}, \nabla h)H.
\end{eqnarray*}
Substituting the above inequality into \eqref{C2-1-1} and using
$H(x_0)\geq 1$, we obtain
\begin{eqnarray*}
0&\geq&-C(h, h^{-1}, \nabla h)AH-f^{\frac{1}{n-2}}C(h, h^{-1},
\nabla h)H -C(h, h^{-1}, \nabla h)\\&&+f^{\frac{1}{n-2}}\sum_{i}b^{ii}H,
\end{eqnarray*}
which implies
\begin{eqnarray*}
f^{\frac{1}{n-2}}\sum_{i}b^{ii}&\leq&C(h, h^{-1}, \nabla h)A+C(h,
h^{-1}, \nabla h) (\max f)^{\frac{1}{n-2}}.
\end{eqnarray*}
Then, it follows from \eqref{b}
\begin{eqnarray*}
H^{\frac{1}{n-2}}\leq C(h, h^{-1}, \nabla h)A+C(h, h^{-1}, \nabla
h)(\max f)^{\frac{1}{n-2}}.
\end{eqnarray*}
We have thus complete the proof.
\end{proof}

\section{The proof of Theorem}

We first recall Lemma 2.2 in \cite{Guan-Li}.

\begin{lemma}\label{GL-1}
Let $f_1, f_2\in C^{1, 1}(\uS)$ be two nonnegative functions satisfying, for some positive constants $a, b, A$, that
\begin{eqnarray*}
af_i \Delta f_i-b|\nabla f_i|^2
\geq-Af_{i}^{2-\frac{1}{n-2}}, \quad \forall x \in \uS.
\end{eqnarray*}
Then $f=f_1+f_2$ satisfies
\begin{eqnarray*}
a f \Delta f-b|\nabla f|^2
\geq-2Af^{2-\frac{1}{n-2}}, \quad \forall x \in \uS.
\end{eqnarray*}
\end{lemma}

Now we begin to prove Theorem \ref{thm-1}.

\begin{proof}
Set $f_{\epsilon}=f+\epsilon$ for positive small $\epsilon$. It
follows from \cite{HZ.Adv.332-2018.57} that we can find an even
function $h_{\epsilon} \in C^{4, \alpha}$ satisfying \eqref{Lp-A}.
Note that $f$ is even and nonzero, so is $f_{\epsilon}$. We know from
Lemma \ref{GL-1} that $f_{\epsilon}$ satisfies either Condition
\ref{CI} or Condition \ref{CII}. Using Lemma \ref{C0-1}, Lemma
\ref{C0-2} and Lemma \ref{C2}, we have
$\{|h_{\epsilon}|_{C^2(\uS)}\}$ is uniformly bounded by some
independent of $\epsilon$. Let $h=\lim_{\epsilon\rightarrow
0}h_{\epsilon}$. Thus, $h$ is a generalized solution to \eqref{Lp-A}
and $h \in C^{1, 1}(\uS)$. Therefore Theorem \ref{thm-1} is proved.
\end{proof}


\end{document}